\documentclass[draft,12pt]{amsart}
\usepackage{amscd}

\usepackage{latexsym}

\usepackage{amsfonts}
\usepackage{amssymb}

\newcommand{\norm}[1]{\lVert#1\rVert}

\newcommand{\CC}{\mathbb{C}}
\newcommand{\B}{\mathcal{B}}
\newcommand{\veps}{\varepsilon}

\DeclareMathOperator{\Range}{Range}

\swapnumbers

\theoremstyle{plain}
\newtheorem{theorem}{Theorem}[section]
\newtheorem{corollary}[theorem]{Corollary}
\newtheorem{lemma}[theorem]{Lemma}
\newtheorem{proposition}[theorem]{Proposition}

\newtheorem{problem}[theorem]{Problem}

\theoremstyle{plain}

\newtheorem{remark}[theorem]{Remark}

\def\R{{\rm I\kern-2ptR}}
\def\N{{\rm I\kern-2ptN}}

\def\emptyset{\mathop{\raisebox{.1ex}{$\not
\mathrel{\raisebox{.1ex}{$\scriptstyle\bigcirc$}}$}}}

\def\qed{\hskip .6em \raise1.8pt\hbox{\vrule height4pt
width6pt depth2pt}}
\def\qedd{\hskip .4em \raise1.8pt\hbox{\vrule height3pt
width5pt depth1.8pt}}
\def\sq{\hskip .6em \raise1.8pt\hbox{\vrule
height4pt width6pt depth2pt}}

\def\leaderfill{\leaders\hbox to 1em{\hss.\hss}\hfill}

\begin{document}

\title[]{Controlling almost-invariant halfspaces\\in both real and complex settings}

\author{Adi Tcaciuc}
\address{Department of Mathematics and Statistics, MacEwan University, Edmonton, Alberta, T5J P2P, Canada }
\email{atcaciuc@ualberta.ca}

\author{Ben Wallis}
\address{Department of Mathematical Sciences, Northern Illinois University, Dekalb, IL 60115}
\email{benwallis@live.com}

\begin{abstract}If $T$ is a bounded linear operator acting on an infinite-dimensional Banach space $X$, we say that a closed subspace $Y$ of $X$ of both infinite dimension and codimension is an almost-invariant halfspace (AIHS) under $T$ whenever $TY\subseteq Y+E$ for some finite-dimensional subspace $E$, or, equivalently, $(T+F)Y\subseteq Y$ for some finite-rank perturbation $F:X\to X$.  We discuss the existence of AIHS's for various restrictions on $E$ and $F$ when $X$ is a complex Banach space.  We also extend some of these and other results in the literature to the setting where $X$ is a real Banach space instead of a complex one.\end{abstract}

\thanks{The authors thank Robert B. Israel for his helpful comments.}

\maketitle
\theoremstyle{plain}
%%%%%%%%%%%%%%%%%%%%%%%%%%%%%%%%%%%%%%%%%%%%%%Introduction

\section{Introduction}

Let $X$ be an infinite-dimensional Banach space and $T\in\B(X)$ a bounded linear operator.  We say that a subspace $Y$ of $X$ is {\bf almost-invariant} under $T$ whenever $TY\subseteq Y+E$ for some finite-dimensional subspace $E$ of $X$.  In this case, $E$ is called an {\bf error} subspace, and its minimum possible dimension is called the {\bf defect} of $Y$ under $T$.  This is a natural weakening of the notion of an {\bf invariant} subspace, that is, an almost-invariant subspace with defect zero.  Observe that every operator admits almost-invariant subspaces, so to make things nontrivial we say that $Y$ is a {\bf halfspace} whenever it is  a norm-closed subspace of $X$ satisfying $\text{dim}(Y)=\text{dim}(X/Y)=\infty$.  It is straightforward to verify a subspace $Y$ is an almost-invariant halfspace (henceforth {\bf AIHS} under $T$ with defect $\leq d$ if and only if it is an invariant halfspace (henceforth {\bf IHS}) under $T+F$ for some operator $F\in\mathcal{B}(X)$ with rank $\leq d$.  We may therefore extend the notion of an AIHS as follows by considering compact perturbations instead of finite-rank ones.  In this case, we say that a subspace $Y$ is {\bf essentially invariant} when it is invariant under $T+K$ for some compact operator $K$, and an {\bf essentially invariant halfspace} (hereafter, {\bf EIHS}) whenever, in addition, $Y$ is a halfspace.  The project of finding an almost-invariant halfspace, henceforth {\bf AIHS}, for every operator acting on an infinite-dimensional Banach space is called the {\bf AIHS problem}.  Similarly, the project of finding an EIHS is called the {\bf EIHS problem}.

AIHS's were first defined in \cite{APTT09}, and have been studied in numerous additional papers since then, namely \cite{Po10}, \cite{MPR13}, \cite{PT13}, \cite{SW14}, \cite{BR15}, and \cite{SW16}.  Although the AIHS problem remains open in its fullest generality, numerous partial results have been obtained.  For example, it was shown in \cite[Theorem 2.7]{PT13} that every operator acting on an infinite-dimensional complex reflexive space admits an AIHS of defect $\leq 1$.  Then, in \cite{SW14} it was shown that every operator $T$ acting on an infinite-dimensional complex Banach space and with at most countably many eigenvalues admits an AIHS, again of defect $\leq 1$, and this was extended in \cite[Remark 2.5]{SW16} to the case where $T$ merely commutes with an operator with at most countably many eigenvalues but which is not finite-rank or a multiple of the identity.

In this paper we further study two kinds of AIHS constructions.  First, in section 2, we discuss how to control the error susbpaces and finite-rank perturbations associated with an AIHS.  In particular we show that for every operator $T$ acting on an infinite-dimensional complex reflexive space $X$ and every $\varepsilon>0$ there exists a finite-rank operator $F\in\mathcal{B}(X)$ with norm $<\varepsilon$ and such that $T+F$ admits an IHS.  This represents an extension of an analogous result for complex infinite-dimensional Hilbert spaces proved in \cite[S3]{PT13}.  Also in section 2 we show that when $X$ is an infinite-dimensional complex Banach space and $T\in\mathcal{B}(X)$ with no eigenvalues and countable spectrum, then for every nonzero $x\in X$ there exists an AIHS under $T$ with error $\subseteq\text{span}\{x\}$.

Second, in section 3, we discuss the existence of AIHS's when $X$ is a real Banach space instead of a complex one.  Indeed, many of the most significant results for AIHS's have been obtained using spectral theory, and their proofs do not translate easily to the case where $X$ is real.  However, in case $\sigma(T)=\sigma(T)\cap\mathbb{R}$ or else $\sigma(T)\cap\mathbb{R}$ is infinite, the operator $T^*$ always admits an AIHS of defect $\leq 1$, and if in addition $X$ is reflexive then so does $T$.

Most of the notation is standard, such as might be found in \cite{AA02} or \cite{Ai04}; however, we shall review a few basic definitions and notations which are essential for our purposes.  First, let us use the conventions $\mathbb{N}=\{1,2,3,\cdots\}$ and $\mathbb{N}_0=\mathbb{N}\cup\{0\}$.  If $X$ and $Y$ are Banach spaces then we denote by $\mathcal{B}(X)$ the algebra of bounded linear operators on $X$.  If $A\subseteq X$ is a set then we denote by $\overline{A}$ the closure of $A$ in $X$ and $[A]$ the closure of its linear span.  Denote by
\[\mathcal{N}(T)=\{x\in X:Tx=0\}\]
the {\bf null space}, or {\bf kernel}, of $T$.  If $Y$ is a linear subspace of $X$ then
\[Y^\perp:=\left\{x^*\in X^*:x^*(y)=0\text{ for all }y\in Y\right\}\]
is called the {\bf annihilator} of $Y$, and is a closed subspace of $X^*$ (even when $Y$ is not closed).  Observe that we always have the relations
\begin{equation}\label{annihlator-relations}\mathcal{N}(T^*)=(\overline{TX})^\perp=(TX)^\perp\;\;\;\text{ and }\;\;\;\overline{T^*X^*}\subseteq\mathcal{N}(T)^\perp.\end{equation}
If $Y$ is a closed subspace of $X$ then we also have
\begin{equation}\label{annihilator-dual-relations}Y^\perp\approx(X/Y)^*\;\;\;\text{ and }\;\;\;Y^*\approx X^*/(Y^\perp),\end{equation}
where each ``$\approx$'' means ``is isomorphic to.''  Note that in the special case where $X=H$ is a Hilbert space, for any closed subspace $Y$ of $H$, the annihilator $Y^\perp$ can be viewed as a closed subspace of $H$, and serves as the orthogonal complement of $Y$ in $H$.

If $X$ is a complex Banach space and $T\in\mathcal{B}(X)$ then we denote by
\[\sigma(T):=\left\{\lambda\in\mathbb{C}:\lambda-T\text{ is not invertible in the algebra }\mathcal{B}(X)\right\}\]
the {\bf spectrum} of $T$ and $\rho(T):=\mathbb{C}\setminus\sigma(T)$ its {\bf resolvent}.  It is well known that $\sigma(T)$ is a compact subset of the complex plane.  Thus, $\rho(T)$ contains a unique unbounded component called the {\bf full resolvent}, and denoted $\rho_\infty(T)$.  Then $\sigma_\infty(T)=\mathbb{C}\setminus\rho_\infty(T)$ is called the {\bf full spectrum} of $T$.  We can also define the {\bf spectral radius}
\[r(T):=\sup_{\lambda\in\sigma(T)}|\lambda|.\]
We shall use other parts of the spectrum as well, including the following.  Let
\[\sigma_p(T):=\left\{\lambda\in\mathbb{C}:\lambda-T\text{ is not injective}\right\}\]
as the {\bf point spectrum}, and
\[\sigma_{su}(T):=\left\{\lambda\in\mathbb{C}:\lambda-T\text{ is not surjective}\right\}\]
the {\bf surjectivity spectrum} of $T$.  Observe that $\sigma_p(T)$ is precisely the set of eigenvalues under $T$, and that $\partial\sigma(T)\subseteq\sigma_{su}(T)$ (cf., e.g., \cite[Theorem 2.42]{Ai04}).

The sets $\sigma(T)$, $\sigma_p(T)$, and $\sigma_{su}(T)$ have, so far, only been defined in the complex setting.  To extend these definitions to the case where $X$ is real, we will need to consider the {\bf complexification} $X_\mathbb{C}$, defined as follows.  Note that $X\oplus iX$ is a vector space under the operations
\[(x_1\oplus iy_1)+(x_2\oplus iy_2)=(x_1+x_2)\oplus i(y_1+y_2)\;\;\;\text{ for all }x_1,x_2,y_1,y_2\in X,\text{ and}\]
\[(a+ib)(x\oplus iy)=(ax-by)\oplus i(bx+ay)\;\;\;\text{ for all }x,y\in X\text{ and }a,b\in\mathbb{R}.\]
We can then define a complete norm
\[\|x\oplus iy\|_{X_\mathbb{C}}=\sup_{\theta\in[0,2\pi]}\|x\cos\theta+i y\sin\theta\|\;\;\;\text{ for all }x,y\in X.\]
The complexification $X_\mathbb{C}$ is defined as the Banach space $X\oplus iX$, under the given vector space operations, endowed with this norm.  Note that in this case $\|\cdot\|_{X_\mathbb{C}}$ satisfies
\begin{equation}\label{complexification-estimate}\frac{1}{2}\left(\|x\|_X+\|y\|_X\right)\leq\|x\oplus iy\|_{X_\mathbb{C}}\leq\|x\|_X+\|y\|_X\end{equation}
for all $x,y\in X$.
%We shall call the set $\{x\oplus i0:x\in X\}\subseteq X_\mathbb{C}$ the {\bf real part} of $X_\mathbb{C}$.  Since the real part of $X_\mathbb{C}$ is isometric to $X$ when viewed as a real subspace, as an abuse of notation we may sometimes identify each $x\in X$ with $x\oplus i0$.
Observe that for any continuous linear operator $T:X\to Y$ there exists a continuous linear {\bf complexification operator} $T_\mathbb{C}:X_\mathbb{C}\to Y_\mathbb{C}$ defined by
\[T_\mathbb{C}(x_1\oplus ix_2)=(Tx_1)\oplus i(Tx_2)\;\;\;\text{ for all }x_1,x_2\in X.\]
Note in particular that $\mathbb{R}_\mathbb{C}=\mathbb{C}$ so that if $f\in X^*$ then $f_\mathbb{C}\in(X_\mathbb{C})^*$.  Please see \cite[S1.1]{AA02} for further details on the complexications of $X$ and $T$.

Now we can extend the various parts of the spectrum by writing
\[\sigma(T):=\sigma(T_\mathbb{C}),\;\;\;\sigma_p(T):=\sigma_p(T_\mathbb{C}),\;\;\;\sigma_{su}(T):=\sigma_{su}(T_\mathbb{C}),\]
and so on.  Let us caution the reader that the definitions of $\sigma(T)$ and $\rho(T)$ in the real Banach space setting {\it do not coincide} with the respective definitions in the complex setting.  In particular, if $T\in\mathcal{B}(X)$ and $X$ is a real Banach space then $\lambda-T$ does not exist when $\lambda\notin\mathbb{R}$, and so of course it could not be is invertible in the algebra $\mathcal{B}(X)$ in this case.  Instead, we shall define the {\bf real spectrum} of $T\in\mathcal{B}(X)$ for a real Banach space $X$ as the set
\[\sigma_\mathbb{R}(T):=\left\{\lambda\in\mathbb{R}:\lambda-T\text{ is not invertible in the algebra }\mathcal{B}(X)\right\}\]
and $\rho_\mathbb{R}(T):=\mathbb{R}\setminus\sigma_{\mathbb{R}}(T)$ its {\bf real resolvent}. Observe that, unlike the spectrum $\sigma(T)$, the real spectrum $\sigma_\mathbb{R}(T)$ could be empty.  It is well-known, and straightforward to verify, that $\sigma_\mathbb{R}(T)=\sigma(T)\cap\mathbb{R}$ for any operator $T$ acting on a real Banach space.%  For more information on the real spectrum, please consult [reference???].

In section 2 we shall be interested in the behavior of the local resolvent mappings, and so we shall also review a few basic facts from local spectral theory.  If $X$ is a complex Banach space and $T\in\mathcal{B}(X)$, we say that $T$ has the \textbf{single-valued extension property} (hereafter, {\bf SVEP}) whenever, for every open set $U\subseteq\CC$, the only analytic function $f:U\to X$ that satisfies equation $(T-\lambda I)f(\lambda)=0$ for all $\lambda\in U$ is the trivial function $f\equiv 0$.  It is not hard to see that if $\sigma_p(T)$ has empty interior then $T$ has SVEP. For any $x\in X$, the \textbf{local resolvent} set $\rho_T(x)$ is defined as the set of all complex numbers $\lambda$ for which there exists an open neighbourhood $U_\lambda$ of $\lambda$ and an analytic function $f:U_\lambda\to X$ such that for any $\mu\in U_\lambda$ we have $(\mu-T)f(\mu)=x$. This local analytic solution is unique for all $x\in X$ if and only if $T$ has SVEP; in this case there is an unique maximal analytic extension of $(\lambda-T)^{-1}x$ from $\rho(T)$ to $\rho_{T}(x)$, called \textbf{local resolvent function} for $T$ at $x$. The \textbf{local spectrum} of $T$ at $x$ is defined as $\sigma_{T}(x)=\CC\setminus\rho_{T}(x)$. Clearly, for any $x\in X$, $x\neq 0$,  $\sigma_{T}(x)$ is a nonempty, compact subset of the plane contained in $\sigma(T)$.  For more details please see either of the excellent reference books \cite{LN00} or \cite{Ai04}.

\section{Controlling errors and finite-rank perturbations}

There are two kinds of control for constructions of AIHS's which we will consider in this section.  First, we could ask for a bound on the norm of the finite rank perturbation.  Brown and Pearcy discussed a similar control for the associated compact perturbations of EIHS's under operators acting on a complex infinite-dimensional Hilbert space $H$.  In particular, they showed in \cite[Theorem 3.1]{BP71} that any $T\in\mathcal{B}(H)$ admits, for any $\varepsilon>0$ a decomposition $H=Y\oplus Y^\perp$ such that
\[T=\begin{bmatrix}\lambda+K&*\\L&*\end{bmatrix}\]
for some $\lambda\in\mathbb{C}$ and such that $K$ and $L$ are compact with norm $<\varepsilon$.  In this case,
\[(T-L)Y=(\lambda+K)Y\subseteq Y\]
so that $Y$ is an EIHS under $T$ with associated compact perturbation of norm $<\varepsilon$.  Several nice extensions of this result were established in \cite[\S3]{PT13}, in particular that, if $\partial\sigma(T)\setminus\sigma_p(T)\neq\emptyset$ then the decomposition $Y\oplus Y^\perp$ above can be chosen so that $K$ is compact, $L$ is rank-1, and they both still have norm $<\varepsilon$.

These results suggest the following open question.

\begin{problem}\label{problem-norm-control}Suppose $X$ is an infinite-dimensional Banach space and $T\in\mathcal{B}(X)$ admits an AIHS.  Does there exist, for every $\varepsilon>0$, a finite-rank operator $F\in\mathcal{B}(X)$ of norm $<\varepsilon$ such that $T+F$ admits an IHS?\end{problem}

\noindent Although this problem is not solved in general, it was shown in \cite[\S3]{PT13} that it is true when $X=H$ is an infinite-dimensional complex Hilbert space.  It turns out that only minor changes are necessary to adapt their proof to the case where $X$ is an infinite-dimensional complex reflexive space.  We will do this in several steps, beginning with the following proposition.

\begin{proposition}\label{small-norm-perturbation}Let $X$ be a (real or complex) infinite-dimensional Banach space, and let $T\in\mathcal{B}(X)$ be such that $0\notin\sigma_p(T)$.  Suppose there is a sequence $(\lambda_n)_{n=1}^\infty\subset\mathbb{C}$ of complex numbers satisfying $\lambda_n\to 0$, and also a sequence $(h_n)_{n=1}^\infty\subset X$ of vectors satisfying $\|h_n\|\to\infty$ and $(\lambda_n-T)h_n=e$ for all $n\in\mathbb{N}$ and some nonzero $e\in X$.  Then for every $\varepsilon>0$ there exists a rank-1 operator $F\in\mathcal{B}(X)$ with $FX\subseteq[e]$ and $\|F\|<\varepsilon$, and such that $T+F$ admits an IHS, say $Y$, spanned by a basic subsequence of $(h_n)_{n=1}^\infty$.  In this case, $TY\subseteq Y+[e]$.\end{proposition}

\begin{proof}It was shown in the proof of \cite[Theorem 2.3]{PT13} that, under the given hypotheses, $(h_n)_{n=1}^\infty$ admits a basic subsequence.  Thus, passing to a subsequence if necessary, we may assume $(\frac{h_n}{\|h_n\|})_{n=1}^\infty$ is basic with constant $\leq K$ for some $K\in[1,\infty)$, and
\[\sum_{n=1}^\infty\frac{1}{\|h_n\|}\leq\frac{\varepsilon}{2K}.\]
Define the functional $f:\text{span}\{h_n\}_{n=1}^\infty\to\mathbb{C}$ by the rule $f(h_n)=1$ for all $n\in\mathbb{N}$.  Then for any $(a_n)_{n=1}^\infty\in c_{00}$ we have
\[|f\left(\sum_{n=1}^\infty a_n\frac{h_n}{\|h_n\|}\right)|\leq\sum_{n=1}^\infty\frac{|a_n|}{\|h_n\|}\leq\sup_{n\in\mathbb{N}}|a_n|\cdot\sum_{n=1}^\infty\frac{1}{\|h_n\|}\leq 2K\|\sum_{n=1}^\infty a_n\frac{h_n}{\|h_n\|}\|\cdot\frac{\varepsilon}{2K}\]
so that $\|f\|\leq\varepsilon$.  By Hahn-Banach we extend $f$ to a continuous linear functional $f\in X^*$ with norm $\leq\varepsilon$, so that the rank-1 operator $F=f\otimes e$ (defined by $Fx=f(x)e$ for $x\in X$) also has norm $\leq\varepsilon$ (as long as $\|e\|=1$, which we may from the beginning assume without loss of generality).  Now,
\[(T+F)h_n=\lambda h_n-(\lambda_n-T)h_n+f(h_n)e=\lambda_nh_n-e+e=\lambda_nh_n\]
so that $Y=[h_{2n}]_{n=1}^\infty$ is an IHS under $T+F$.\end{proof}

\begin{remark}
  In the previous proposition, note that, by passing to further subsequences of $(h_n)_{n=1}^\infty$, we can obtain an infinite chain $Y_1\supset Y_2\supset Y_3\supset \dots$ of distinct half-spaces of $X$, such that all are almost-invariant with defect $[x]$, whose associated rank-1 perturbations have norms tending to zero.
\end{remark}

\begin{corollary}\label{small-2.3}Let $X$ be an infinite-dimensional complex Banach space and let $T\in\mathcal{B}(X)$.  If $\partial\sigma(T)\setminus\sigma_p(T)\neq\emptyset$ then for every $\varepsilon>0$ there exists an AIHS, say $Y$, under $T$, such that $(T+F)Y\subseteq Y$ for some rank-$1$ operator $F\in\mathcal{B}(X)$ with $\|F\|<\varepsilon$.\end{corollary}

\begin{proof}Shifting $T$ by a scalar if necessary, we may assume $0\in\partial\sigma(T)\setminus\sigma_p(T)$.  Select any $(\lambda_n)_{n=1}^\infty\subset\rho(T)$ satisfying $\lambda_n\to 0$.  As in the proof of \cite[Theorem 2.3]{PT13}, it follows that $\|(\lambda_n-T)^{-1}\|\to\infty$ and hence, by the uniform boundedness principle, that also $\|(\lambda_n-T)^{-1}e\|\to\infty$ for some nonzero $e\in X$.  Now set $h_n=(\lambda_n-T)^{-1}e$ for $n\in\mathbb{N}$ and apply Proposition \ref{small-norm-perturbation}\end{proof}

To prove the main theorem we need two earlier results.  The first of these was implicitly (but not explicitly) given within the proof of \cite[Theorem 2.7]{PT13} (see also Lemma \ref{infinitely-many-eigenvalues} below).

\begin{proposition}\label{infinitely-many-eigenvalues}Let $X$ be an infinite-dimensional complex Banach space and $T\in\mathcal{B}(X)$.  If $\sigma_p(T)$ and $\sigma_p(T^*)$ both have infinite cardinality then $T$ admits an IHS.\end{proposition}

\begin{theorem}[{\cite[Proposition 1.7]{APTT09}}]\label{dual-AIHS}Let $X$ be an infinite-dimensional (real or complex) Banach space, and let $T\in\mathcal{B}(X)$.  If $T$ admits an AIHS of defect $\leq d$, $d\in\mathbb{N}$, then so does $T^*$.\end{theorem}

The next result represents our main tool for attacking Problem \ref{problem-norm-control}.  To prove it, we need a pair of definitions.  If $X$ and $Y$ are Banach spaces then a bounded linear operator $T:X\to Y$ is called {\bf Fredholm} just in case $\mathcal{N}(T)$ and $X/TX$ are both finite-dimensional.  If $X$ is a complex Banach space and $T\in\mathcal{B}(X)$, we define the {\bf essential spectrum} as
\[\sigma_{ess}(T):=\left\{\lambda\in\mathbb{C}:\lambda-T\text{ is not Fredholm}\right\}.\]
Note that $\sigma_{ess}(T)$ is a nonempty compact subset of $\sigma(T)$ (cf., e.g., \cite[Lemma 7.40]{AA02}).

\begin{theorem}\label{dual-norm-control}Let $X$ be a complex infinite-dimensional Banach space and $T\in\mathcal{B}(X)$.  Then there exists $d\in\mathbb{N}$ such that for every $\varepsilon>0$ there is an operator $F\in\mathcal{B}(X^*)$ of rank $\leq d$ satisfying $\|F\|<\varepsilon$, and such that $T^*+F$ admits an IHS.\end{theorem}

\begin{proof}We follow the proof to \cite[Proposition 3.4]{PT13}.  By Proposition \ref{infinitely-many-eigenvalues} we may assume that either $T$ or $T^*$ has finitely many eigenvalues.  Observe that if $d\in\mathbb{N}$ and $T+F$ admits an IHS for some rank-$d$ operator $F\in\mathcal{B}(X)$ then $T^*+F^*$ admits an IHS by Theorem \ref{dual-AIHS}, where $\|F^*\|=\|F\|$.  In this way, if the conclusion holds for $T$ then it holds for $T^*$.  We may therefore assume without loss of generality that $\sigma_p(T)$ is finite.  We may also assume, this time via Corollary \ref{small-2.3}, that $(\partial\sigma(T))\setminus\sigma_p(T)=\emptyset$.  As $\sigma_p(T)$ is finite, this means $\sigma(T)$ is finite as well.  By a standard spectral projections argument (see \cite[Remark 2.9]{MPR12}) we pass to a shifted restriction to a finite-codimensional $T$-invariant subspace if necessary, so that $\sigma(T)$ is quasinilpotent.

Clearly, we may assume that $m:=\text{dim}(X/\overline{TX})$ and $n:=\text{dim}(\mathcal{N}(T))$ are both finite, else $T$ admits an IHS.  Let us therefore decompose $X=\overline{TX}\oplus E$ and $X=W\oplus\mathcal{N}(T)$ for an $m$-dimensional subspace $E$ and a closed, $n$-codimensional subspace $W$.  Observe that by \eqref{annihlator-relations}, $\mathcal{N}(T^*)=(\overline{TX})^\perp$ and $\overline{T^*X^*}\subseteq\mathcal{N}(T)^\perp$, and by \eqref{annihilator-dual-relations}, $(\overline{TX})^\perp\approx(X/\overline{TX})^*$ and $\mathcal{N}(T)^*\approx X^*/(\mathcal{N}(T)^\perp)$.  Hence, if $m\leq n$ we have
\begin{multline*}n^*:=\text{dim}(\mathcal{N}(T^*))=\text{dim}((\overline{TX})^\perp)=\text{dim}(X/\overline{TX})=m\leq n
\\=\text{dim}(\mathcal{N}(T))=\text{dim}(X^*/(\mathcal{N}(T)^\perp))\leq\text{dim}(X^*/\overline{T^*X^*})=:m^*.\end{multline*}
By passing to $T^*$ if necessary, we may therefore assume that $n\leq m$.  However, note that we cannot pass to $T^*$ a second time, and that now we must find an AIHS for $T$ whose associated finite-rank perturbation has norm $<\varepsilon$.

Set $d:=m+1$, and pick any $\varepsilon>0$.  Let $\{e_1,\cdots,e_m\}$ be a basis for $E$ and $\{g_1,\cdots,g_n\}$ a basis for $\mathcal{N}(T)$.  Since $X=W\oplus\mathcal{N}(T)$ we may define the operator $F\in\mathcal{B}(X)$ by the rule $Fg_i=\delta e_i$ for $i=1,\cdots,n$ and $\delta>0$, and $Fw=0$ for all $w\in W$.  By making $\delta$ smaller if necessary we may assume $\|F\|<\varepsilon/2$.  Notice that $0\notin\sigma_p(T+F)$, as if $w\in W$, $g\in\mathcal{N}(T)$, and $0=(T+F)(w+g)=Tw+Fg$ then the complementation of $\overline{TX}$ and $E$ implies that $Tw=0$ and $Fg=0$, which is true only if $w=0$ and $g=0$.

It is well-known that the essential spectrum is stable under compact perturbations (cf., e.g., \cite[Corollary 4.47]{AA02}).  In particular, $\sigma_{ess}(T+F)=\{0\}$.  It is also known that when the essential spectrum of an operator is $\{0\}$, its spectrum is at most countable (cf., e.g., \cite[Corollary 7.50]{AA02}).  So, $0\in(\partial\sigma(T+F))\setminus\sigma_p(T+F)$.  By Corollary \ref{small-2.3} we now obtain $G\in\mathcal{B}(X)$ of rank 1 and norm $<\varepsilon/2$ such that $T+F+G$ admits an IHS.\end{proof}

\begin{corollary}Let $X$ be an infinite-dimensional reflexive space, and let $T\in\mathcal{B}(X)$.  Then there exists $d\in\mathbb{N}$ such that for every $\varepsilon>0$ there is an operator $F\in\mathcal{B}(X)$ of rank $\leq d$ satisfying $\|F\|<\varepsilon$, and such that $T+F$ admits an IHS.\end{corollary}

Problem \ref{problem-norm-control} remains open in its fullest generality.  However, we should mention that a weaker result was obtained but not explicitly stated in \cite[\S4]{SW16}, wherein it was shown that if $X$ is an infinite-dimensional complex Banach space then for every $\varepsilon>0$ there exists a nuclear operator $N\in\mathcal{B}(X)$ of norm $<\varepsilon$ such that $T+N$ admits an AIHS of defect $\leq 1$.  Consequently, the set
\[\left\{T\in\mathcal{B}(X):T\text{ admits an AIHS of defect }\leq 1\right\}\]
is norm-dense in $\mathcal{B}(X)$ (\cite[Corollary 4.3]{SW16}).  Let us give an analogous result in the reflexive case.

\begin{corollary}Let $X$ be a complex infinite-dimensional reflexive space.  Then the set
\[\left\{T\in\mathcal{B}(X):T\text{ admits an IHS}\right\}\]
is norm-dense in $\mathcal{B}(X)$.\end{corollary}

For the remainder of this section, let us consider a different kind of control on AIHS's, as follows.

\begin{problem}\label{problem-PE}Let $X$ be an infinite-dimensional Banach space and suppose $T\in\mathcal{B}(X)$ admits an AIHS.  Does there exists, for any finite-dimensional subspace $E$ of $X$, an AIHS $Y$ with error $\subseteq E$?\end{problem}

We give a counter-example to show that the above question has a negative answer in the general case.  Just consider a ``Donoghue operator'' $D\in\mathcal{B}(\ell_2)$ defined by the rules $De_{n+1}=2^{-n}e_n$ for $n\in\mathbb{N}$ and $De_1=0$, where $(e_n)_{n=1}^\infty$ is the canonical basis for $\ell_2$.  This is a compact, quasinilpotent operator which admits an AIHS of defect $\leq 1$ (by \cite[Corollary 3.5]{APTT09} in the complex case and by Corollary \ref{real-reflexive} below in the real case).  On the other hand, for any $N\in\mathbb{N}$ we can find an $N$-dimensional subspace $E_N=[e_1,\cdots,e_N]$ such that no AIHS under $D$ admits an error $\subseteq E_N$.  To see this, let us suppose towards a contradiction that $Y$ is an AIHS under $D$ with error $\subseteq E_N$.  Write $W_N=[e_n]_{n=N+1}^\infty$ so that there exist natural projections $P_{E_N},P_{W_N}\in\mathcal{B}(\ell_2)$ such that $P_{E_N}\ell_2=E_N$, $P_{W_N}\ell_2=W_N$, and $P_{E_N}+P_{W_N}=1$ (the identity on $\ell_2$).  Observe that $P_{W_N}Y$ is a halfspace in $W_N$, and that
\begin{multline*}(P_{W_N}D|_{W_N})(P_{W_N}Y)\subseteq(P_{W_N}D)(Y+P_{E_N}Y)=(P_{W_N}D)Y+P_{W_N}DE_N\\=(P_{W_N}D)Y\subseteq P_{W_N}(Y+E_N)=P_{W_N}Y\end{multline*}
so that $P_{W_N}Y$ is an IHS under the operator $P_{W_N}D|_{W_N}\in\mathcal{B}(W_N)$.  However, $P_{W_N}D|_{W_N}$ is itself a Donoghue operator, which contradicts the fact, proved in \cite[Theorem 4.12]{RR03}, that the only nontrivial invariant subspaces of Donoghue operators are finite-dimensional.

Nevertheless, we can still establish some partial affirmative results regarding Problem \ref{problem-PE}.

%In this section we show that any bounded operator in a separable Banach space which has a countable spectrum and no eigenvalues, admits a rank one perturbation with a prescribed range, which has an invariant half-space.

\begin{theorem} \label{main1}
%Let $X$ be a separable Banach space and $T\in\B(X)$ a bounded operator such that $\sigma(T)$ is countable and $\sigma_p(T)=\emptyset$. Then for any non-zero $x\in X$, there exists a a half-space $Y$ such that $TY\subseteq Y+[x]$.
%Let $X$ be an infinite-dimensional complex Banach space and $T\in\B(X)$ a bounded operator such that $\sigma(T)$ is countable and $\sigma_p(T)=\emptyset$. Then for any non-zero $x\in X$, there exists a a half-space $Y$ such that $TY\subseteq Y+[x]$.
Let $X$ be an infinite-dimensional complex Banach space and $T\in\B(X)$ a bounded operator such that $\sigma(T)$ is countable and $\sigma_p(T)=\emptyset$. Then for any nonzero $x\in X$ and any $\varepsilon>0$ there exists $F\in\B(X)$ with $\norm{F}<\veps$ and $\Range(F)=[x]$, and such that $T-F$ admits an IHS, say $Y$.  In this case, $TY\subseteq Y+[x]$.\end{theorem}

\begin{proof}Note first that since $\sigma_p(T)=\emptyset$, $T$ has SVEP. Fix $e\in X$, $x\neq 0$, and consider the $f:\rho_{T}(x)\to\CC$ the local resolvent function for $T$ at $x$. The local spectrum $\sigma_{T}(x)=\CC\setminus\rho_{T}(x)$ is non-empty, compact, and, from the hypotheses, countable, therefore it contains at least one isolated point, which by shifting if necessary is zero. Since $f$ is the unique \emph{maximal} extension of $\lambda\in\rho(T)\mapsto(T-\lambda)^{-1}x$, it follows that any isolated point of  $\sigma_{T}(x)$ is a non-removable singularity of $f$.  Hence, $f$ is unbounded near the isolated points of the local spectrum, so that we can find a sequence $(\lambda_n)_{n=1}^\infty\subset\rho_{T}(x)$ satisfying $\lambda_n\to 0$ and $\norm{f(\lambda_n)}\to\infty$.  Using $h_n=f(\lambda_n)$ and $e=x$, we may now apply Proposition \ref{small-norm-perturbation}.\end{proof}

In the case when $X$ is reflexive, we also have the dual statement of the previous result, in the sense that we can control for the kernel of the perturbation.
\begin{theorem}
 Let $X$ be a  reflexive  Banach space and $T\in\B(X)$ a bounded operator such that $\sigma(T^{*})$ is countable and $\sigma_p(T^{*})=\emptyset$. Then for any hyperplane $H\subseteq X$,  and for any $\veps>0$, there exists $F\in\B(X)$ with $\norm{F}<\veps$ and $\mathcal{N}(F)=H$, such that $T-F$ admits an IHS.
\end{theorem}

\begin{proof}
Fix $H\subseteq X$ a hyperplane, and $\veps>0$. Let $x^{*}\in X^{*}$ such that $\mathcal{N}(x^{*})=H$. From Theorem \ref{main1} we can find $G\in \B(X^{*})$ with $\Range(G)=[x^{*}]$ and  $\norm{G}<\veps$, such that $T^{*}-G$ has an IHS. Then it follows from Theorem \ref{dual-AIHS} that $T^{**}-G^{*}$ has an IHS as well.  Since $X$ is reflexive, we can find $F\in\B(X)$ such that $F^{*}=G$ and $T-F$ has an IHS. Clearly, $\mathcal{N}(F)=\mathcal{N} ({x^{*}})=H$ and $\norm{F}=\norm{G}<\veps$.
\end{proof}

The next theorem shows that any rank one perturbation either has an invariant subspace, or we can find arbitrarily close perturbations that have invariant half-spaces.

  \begin{theorem}
   Let $X$ be a  Banach space and $T\in\B(X)$ a bounded operator such that $\sigma(T)$ is countable.  Then for any rank-1 operator $F$, either $T+F$ has eigenvalues, or else for any $\veps>0$ there exists a rank-1 operator $G$ with $\|F-G\|<\veps$ and such that $T+G$ admits an IHS.
  \end{theorem}
  
   \begin{proof}
   Note first that if $\sigma(T)$ is countable, then for any compact $K$ we have that $\sigma(T+K)$ is countable as well. Indeed, we have $\sigma_{ess}(T+K)=\sigma_{ess}(T)\subseteq\sigma(T)$, so $\sigma_{ess}(T+K)$ is countable. Now consider $A:=\CC\setminus\sigma_{ess}(T+K)$.  Since $\sigma_{ess}(T+K)$ is countable it follows that $A$ is connected.  In \cite[Theorem III.19.4]{Mu07} it was shown that, whenever $A$ is a connected and unbounded component of $\mathbb{C}\setminus\sigma_{ess}(T+K)$, the set $A\cap\sigma(T+K)$ consists of at most countably many isolated points.  In our case, $A\cap\sigma(T+K)=\sigma(T+K)\setminus\sigma_{ess}(T+K)$, and therefore
\[\sigma(T+K)=(\sigma(T+K)\setminus\sigma_{ess}(T+K))\cup\sigma_{ess}(T+K)\]
is also countable as claimed.

   Fix $F\in\B(X)$ a rank one operator and assume $\sigma_p(T+F)=\emptyset$. Let $\veps>0$ arbitrary. Since $\sigma(T)$ is countable it follows from the previous paragraph that $\sigma(T+F)$ is countable as well, and we can apply Theorem \ref{main1}. Therefore we can find a rank one bounded operator $G_{\veps}$ with $\Range(G_\veps)=\Range(F)$ and $\norm{G_{\veps}}<\veps$, such that $T+F-G_{\veps}$ has an IHS. Set $G:=F-G_{\veps}$. Then $G$ has rank-1 since $F$ and $G_{\veps}$ have the same 1-dimensional range, and $\norm{F-G}=\norm{G_{\veps}}<\veps$.
    \end{proof}

\section{AIHS's in the real setting}

In this section we extend some of the results in \cite{APTT09}, \cite{MPR13}, \cite{PT13}, \cite{SW14}, and \cite{SW16} to the real Banach space setting.  Indeed, most research into the AIHS problem has relied heavily on spectral theory, and consequently many results have only been established for complex Banach spaces.  For instance, almost all the proofs in \cite{PT13} and \cite{SW14} assume that the underlying Banach space is complex.

We remark that most of the AIHS's are constructed similarly to the following.  Let $e\in X$ be a nonzero vector and let $(\lambda_n)_{n=1}^\infty\subseteq\rho(T)$ be a sequence of complex numbers so that $h_n:=(\lambda_n-T)^{-1}e$ are well-defined.  Then any subsequence of $(h_n)_{n=1}^\infty$ spans an almost-invariant subspace $Y$ under $T$ with error $\subseteq[e]$.  To get an AIHS in the complex setting, then, it remains only to show that the $h_n$'s span a halfspace.

In principle, the above construction might sometimes work even in the real setting.  Namely, if the $\lambda_n$'s are all real, then each $\lambda_n-T$ is invertible as a real operator, and hence $h_n=(\lambda_n-T)^{-1}e$ forms a sequence of vectors in the real Banach space $X$ which span an almost-invariant subspace under the real operator $T$.

Let us therefore begin this section with the following theorem.  It has a proof almost identical to Corollary \ref{small-2.3}, and hence we omit it.%  The only noticeable difference is that we must use $\sigma_\mathbb{R}(T)$ in place of $\sigma(T)$ when $T$ acts on a real Banach space.

\begin{theorem}\label{boundary-real-spectrum}Let $X$ be an infinite-dimensional real Banach space and $T\in\mathcal{B}(X)$.  If $(\partial\sigma_\mathbb{R}(T))\setminus\sigma_p(T)\neq\emptyset$ then $T$ admits an AIHS $Y$ of defect $\leq 1$.  Furthermore, for any $\varepsilon>0$ we can choose $Y$ to be invariant under $T+F$ for some rank-1 operator $F\in\mathcal{B}(X)$ with $\|F\|<\varepsilon$.\end{theorem}

\begin{remark}We regard $\sigma_\mathbb{R}(T)$ as a subspace of $\mathbb{R}$ so that $\partial\sigma_\mathbb{R}(T)$ may not always coincide with $\partial\sigma(T)\cap\mathbb{R}$.  For example, if $\sigma(T)=[0,1]=\{\lambda\in\mathbb{R}:0\leq\lambda\leq 1\}$ then $\partial\sigma(T)\cap\mathbb{R}=\partial\sigma(T)=[0,1]$ whereas $\partial\sigma_\mathbb{R}(T)=\{0,1\}$.\end{remark}

Next, we recall that a sequence $(x_n)_{n=1}^\infty$ is in a Banach space $X$ is called a {\bf minimal system} whenever there exists a sequence $(x_n^*)_{n=1}^\infty$ of bounded linear functionals in $X^*$ which serve as respective coordinate functionals for the space $\text{span}(x_n)_{n=1}^\infty$.  Equivalently, $(x_n)_{n=1}^\infty$ is minimal whenever $x_N\notin[x_n]_{n=N+1}^\infty$ for every $N\in\mathbb{N}$ (cf., e.g., \cite[Lemma 2.6]{PT13}).
  In \cite[Theorem 3.2]{APTT09} and \cite[Remark 1.3]{MPR13} it was shown that in the complex setting, if there exists $e\in X$ such that $(T^ne)_{n=0}^\infty$ forms a minimal system, and $\rho_\infty(T)$ (the unique unbounded component of $\rho(T)$) contains a punctured neighborhood of zero, then $T$ admits an AIHS of defect $\leq 1$.  To prove it, the authors constructed a nonconstant entire function $F:\mathbb{C}\to\mathbb{C}$ and used Picard's Great Theorem to find a sequence of distinct complex numbers $(\lambda_n)_{n=1}^\infty$ with $F(\lambda_n)=0$ for every $n\in\mathbb{N}$.  By the Identity Theorem $|\lambda_n|\to\infty$ so that, passing to a subsequence if necessary, it may be assumed that $\lambda_n^{-1}\in\rho_\infty(T)$ for every $n\in\mathbb{N}$.  That $(T^ne)_{n=0}^\infty$ is minimal ensures that $y_n:=(\lambda_n^{-1}-T)^{-1}e$ forms a linearly independent sequence, and the fact that $F(\lambda_n)=0$ for all $n\in\mathbb{N}$ was then used to construct a linearly independent sequence $(f_k)_{k=0}^\infty\subseteq X^*$ of continuous linear functionals annihilating $(y_n)_{n=1}^\infty$.  It follows that $Y:=[y_n]_{n=1}^\infty$ is an AIHS with error $\subseteq[e]$.

This result can be adapted without compromise to the real setting.  Most of the proof follows that of the complex setting, only taking care to distinguish when values are real, complex, when vectors lie in $X$ versus $X_\mathbb{C}$, and when functionals lie in $X^*$ versus $(X_\mathbb{C})^*$.  The only significant difficulty is that we cannot use Picard's Great Theorem, because the sequence $(\lambda_n)_{n=1}^\infty$ needs to consist of real values when $X$ is a real Banach space.  Instead, we will use the following lemma.  Note that this allows us to extend the result for the complex setting by substituting the condition that $(-\veps,0)\subseteq\rho_\infty(T)$ for some $\veps>0$ in place of $\rho_\infty(T)$ containing a punctured neighborhood of zero.

\begin{lemma}\label{real-analytic}Fix a sequence $(a_n)_{n=0}^\infty$ of positive real numbers.  Then there exist a sequence $(b_n)_{n=0}^\infty$ of positive real numbers and a sequence $(c_n)_{n=0}^\infty$ of negative real numbers satisfying the following.
\begin{itemize}\item[(i)]  $\lim_{n\to\infty}\sqrt[n]{b_n}=0$;
\item[(ii)] $0<b_n\leq a_n$ for all $n\in\mathbb{N}_0$;
\item[(iii)]  $F(c_n)=0$ for all $n\in\mathbb{N}_0$, where $F:\mathbb{C}\to\mathbb{C}$ is an entire function defined by the rule $F(z)=\sum_{n=0}^\infty b_nz^n$; and
\item[(iv)]  $\displaystyle\lim_{n\to\infty}c_n=-\infty$.\end{itemize}\end{lemma}

\begin{proof}We thank Robert Israel for suggesting this proof.

Without loss of generality we may assume $\lim_{n\to\infty}\sqrt[n]{a_n}=0$ so that (i) will follow immediately from (ii).  For notation we write $p_n(z)=\sum_{j=0}^{2n+1}b_jz^j$, so that $p_n:\mathbb{C}\to\mathbb{C}$ and $p_n|_\mathbb{R}:\mathbb{R}\to\mathbb{R}$ are polynomials for each $n\in\mathbb{N}$

Let us define the sequences $(b_n)_{n=0}^\infty$, $(r_n)_{n=0}^\infty\subseteq(-\infty,0)$, and $(\varepsilon_n)_{n=1}^\infty\subseteq(0,\infty)$ inductively so that the following conditions are satisfied for each $N\in\mathbb{N}$:
\begin{itemize}%\item[(1)]  $0<b_N\leq a_N$;
\item[($*$)]  $p_N(r_j-\varepsilon_j)<0<p_N(r_j+\varepsilon_j)$ for all $j\leq N$; and
\item[($**$)]  $r_{N+1}+\varepsilon_{N+1}<r_N-\varepsilon_N$.\end{itemize}

For the base case, set $b_0:=a_0$ and $b_1:=a_1$.  Observe that $p_0(x)=b_0+b_1x$ has exactly one negative real root $r_0:=-\frac{b_0}{b_1}\in(-\infty,0)$, and that there is $\varepsilon_0>0$ with $p_0(r_0-\varepsilon_0)<0<p_0(r_0+\varepsilon_0)$.

For the inductive step, suppose $b_0,\cdots,b_{2n+1}$ are defined for some $n\in\mathbb{N}$, such that ($*$) holds for that $n$.  Select $0<b_{2n+2}\leq a_{2n+2}$ so that
\[b_{2n+2}<\frac{-p_n(r_j-\varepsilon_j)}{(r_j-\varepsilon_j)^{2n+2}}\;\;\;\forall j\leq n.\]
This is possible since all the (finitely many) quantities on the right of the above inequality are positive.  Then
\[p_n(r_j-\varepsilon_j)+b_{2n+2}(r_j-\varepsilon_j)^{2n+2}<0<p_n(r_j+\varepsilon_j)+b_{2n+2}(r_j+\varepsilon_j)^{2n+2}\]
for all $j\leq n$.  Note that
\[\lim_{x\to-\infty}\left(p_n|_\mathbb{R}(x)+b_{2n+2}x^{2n+2}\right)=\infty\]
so that we can select $t_n<r_n-\varepsilon_n$ with $p_n(t_n)+b_{2n+2}t_n^{2n+2}>0$.  Similar to before, we select $0<b_{2n+3}\leq a_n$ so that
\[p_{n+1}(r_j-\varepsilon_j)<0<p_{n+1}(r_j+\varepsilon_j)\]
for all $j\leq n$, and also so that $p_{n+1}(t_n)>0$.  This immediately satisfies ($*$) for $N=n+1$.  Let $r_{n+1}$ denote the most negative real root of $p_{n+1}$.  As $r_{n+1}<t_n<r_n-\varepsilon_n$ and $\lim_{x\to-\infty}p_{n+1}|_\mathbb{R}(x)=-\infty$, we can find $\varepsilon_{n+1}>0$ small enough so that $r_{n+1}+\varepsilon_{n+1}<r_n-\varepsilon_n$ and
\[p_{n+1}(r_{n+1}-\varepsilon_{n+1})<0<p_{n+1}(r_{n+1}+\varepsilon_{n+1}).\]
This satisfies ($**$) for $N=n$, and the inductive step is complete.

So far, we have only proved (i) and (ii).  To prove (iii), observe that, due to (i), $F(z)=\sum_{n=0}^\infty b_nz^n$ is entire.  Also, as
\[F(z)=\lim_{n\to\infty}p_n(z)\;\;\;\forall\,z\in\mathbb{C}\]
we get
\[F(r_n-\varepsilon_n)\leq 0\leq F(r_n+\varepsilon_n)\]
for all $n\in\mathbb{N}_0$.  By the Intermediate Value Theorem applied to the real analytic function $F|_\mathbb{R}:\mathbb{R}\to\mathbb{R}$, for each $n\in\mathbb{N}_0$ there is $c_n\in[r_n\pm\varepsilon_n]$ with $F(c_n)=0$.  This proves (iii).

Finally, let us prove (iv).  Due to ($**$), it follows that
\[-\infty<\cdots<c_3<c_2<c_1<c_0<0.\]
As $F$ is entire and nonconstant, by the Identity Theorem we must have $|c_i|\to\infty$.\end{proof}

Before proving our next main theorem, we need a basic result about complexifications of real Banach spaces.  We include a short proof in lieu of a reference.

\begin{proposition}\label{complexification-functionals}Let $X$ be an infinite-dimensional real Banach space, and let $(x_n)_{n=1}^\infty$ be a minimal system in $X$.  Then $(x_n\oplus i0)_{n=1}^\infty$ is a minimal system in $X_\mathbb{C}$, and if $(x_n^*)_{n=1}^\infty\subseteq(X_\mathbb{C})^*$ are their respective coordinate functionals then $(\text{Re }x_n^*\circ J_\mathbb{R})_{n=1}^\infty$ are the respective coordinate functionals in $X^*$ for $(x_n)_{n=1}^\infty$, where $J_\mathbb{R}:X\to X_\mathbb{R}$ is the natural isometric injection defined by $J_\mathbb{R}x=x\oplus i0$ for $x\in X$.\end{proposition}

\begin{proof}If $(x_N\oplus i0)\in[x_n\oplus i0]_{n=N+1}^\infty$ for any $N\in\mathbb{N}$ then due to the estimate
\begin{multline*}\|(x_N\oplus i0)-\sum_{n=N+1}^\infty(a_n+ib_n)(x_n\oplus i0)\|_{X_\mathbb{C}}
\\=\|(x_N\oplus i0)-\sum_{n=N+1}^\infty(a_nx_n\oplus ib_nx_n)\|_{X_\mathbb{C}}
\geq\frac{1}{2}\|x_N-\sum_{n=N+1}^\infty a_nx_n\|_X\end{multline*}
for all $(a_n)_{n=1}^\infty,(b_n)_{n=1}^\infty\in c_{00}(\mathbb{R})$ (where the last inequality follows by \eqref{complexification-estimate}), we would then have $x_N\in[x_n]_{n=N+1}^\infty$, which cannot be true given that $(x_n)_{n=1}^\infty$ is a minimal system.  This shows that $(x_n\oplus i0)_{n=1}^\infty$ is a minimal system in $X_\mathbb{C}$.

Next, it is clear that $(\text{Re }x_n^*\circ J_\mathbb{R})_{n=1}^\infty$ is a sequence in $X^*$.  To see that they are in fact the coordinate functionals for $(x_n)_{n=1}^\infty$, notice that
\[(\text{Re }x_n^*\circ J_\mathbb{R})(x_m)=\text{Re }x_n^*(x_m\oplus i0)=\text{Re }\delta_{m,n}=\delta_{m,n},\]
for all $m,n\in\mathbb{N}$, where $\delta_{m,n}$ is the Kronecker delta.\end{proof}

Now we give one of the main results of the present section.

\begin{theorem}\label{minimal-orbit}Let $X$ be an infinite-dimensional real or complex Banach space, and let $T\in\mathcal{B}(X)$.  Suppose that some orbit $(T^ne)_{n=0}^\infty$, $e\in X$, forms a minimal system, and that
\[(-\veps,0)\subseteq\rho_\infty(T)\]
for some $\veps>0$. Then $T$ admits an AIHS with error $\subseteq[e]$.\end{theorem}

\begin{proof}We will prove the real case, as the complex case is proved via a simplified argument.

Let $J_\mathbb{R}:X\to X_\mathbb{C}$ denote the natural isometric injection defined by $J_\mathbb{R}x=x\oplus i0$ for $x\in X$.  By Proposition \ref{complexification-functionals}, $((T^ne)\oplus i0)_{n=0}^\infty$ is a minimal system in $X_\mathbb{C}$ with coordinate functionals $(x_n^*)_{n=0}^\infty\subseteq(X_\mathbb{C})^*$, and $(\text{Re }x_n^*\circ J_\mathbb{R})_{n=0}^\infty\subseteq X^*$ are the coordinate functionals for $(T^ne)_{n=0}^\infty$.  Set
\[r_k:=
%\norm{x_k^*}_{(X_\mathbb{C})^*}\vee
\|\text{Re }x_i^*\circ J_\mathbb{R}\|_{X^*}\;\;\;\text{ for each }k\in\mathbb{N}_0.\]
By \cite[Lemma 3.1]{APTT09} there exists a sequence $(a_i)_{i=0}^\infty$ of positive real numbers such that $\sum_{i=0}^\infty a_ir_{i+k}<\infty$ for every $k\in\mathbb{N}_0$.

Observe that by Lemma \ref{real-analytic} we can find an entire nonconstant function
\[F:z\in\mathbb{C}\mapsto\sum_{n=0}^\infty c_iz^i\in\mathbb{C}\]
such that $\beta_k:=\sum_{i=0}^\infty c_ir_{i+k}<\infty$ for all $k\in\mathbb{N}_0$ and $F(\lambda_i)=0$ for some sequence $(\lambda_n)_{n=1}^\infty$ of distinct negative real numbers satisfying $\lambda_n^{-1}\to 0$.  Pass to a subsequence if necessary so that $(\lambda_n^{-1})_{n=1}^\infty\subseteq\rho_\infty(T)$.

For each $n\in\mathbb{N}$ we set $y_n:=(\lambda_n^{-1}-T)^{-1}e$.  Then the space $Y:=[y_n]_{n=1}^\infty$ is almost-invariant with error $\subseteq[e]$.  Note that due to $(T^ne)_{n=0}^\infty$ being minimal, the sequence $(y_n)_{n=1}^\infty$ is linearly independent by \cite[Remark 1.3]{MPR13}, and so $\dim(Y)=\infty$.  To complete the proof it remains only to show that $\dim(X/Y)=\infty$.  To do this, it is enough to construct a linearly independent sequence $(f_n)_{n=1}^\infty\subseteq X^*$ such that each $f_n$ vanishes on $Y$.

Fix $k\in\mathbb{N}_0$.  Since $(T^ie)_{i=0}^\infty$ is minimal, it is linearly independent. We may therefore define a real-valued linear functional $f_k$ on $\text{span}_\mathbb{R}\{T^ie\}_{i=0}^\infty$ by the rule
\[f_k(T^ie)=\left\{\begin{array}{ll}0&\text{ if }i<k,\text{ and}\\c_{i-k}&\text{ if }i\geq k.\end{array}\right.\]
For any $x\in\text{span}_\mathbb{R}\{T^ie\}_{i=0}^\infty$ we have $x=\sum_{i=0}^n\text{Re }x_i^*(J_\mathbb{R}x)T^ie$ for some $n\in\mathbb{N}_0$.  Assume without loss of generality that $n\geq k$ so that we now get
\begin{multline*}
|f_k(x)|
=|f_k\left(\sum_{i=0}^n\text{Re }x_i^*(J_\mathbb{R}x)T^ie\right)|
\\=|\sum_{i=0}^n\text{Re }x_i^*(J_\mathbb{R}x)f_k(T^ie)|
=|\sum_{i=k}^n\text{Re }x_i^*(J_\mathbb{R}x)c_{i-k}|
\\\leq\sum_{i=k}^n\|\text{Re }x_i^*\circ J_\mathbb{R}\|_{X^*}\|x\|_Xc_{i-k}
=\left(\sum_{i=k}^nr_ic_{i-k}\right)\norm{x}_X
=\beta_k\norm{x}_X
.\end{multline*}
Thus $f_k$ is bounded, which means it can be extended to $f_k\in X^*$ by Hahn-Banach.

%Define an entire function $F_k:\mathbb{C}\to\mathbb{C}$ by the rule $F_k(z)=z^kF(z)$.  Observe that\[F_k(z)=\sum_{j=k}^\infty c_{j-k}z^j=\sum_{j=0}^\infty f_k(T^je)z^j=\sum_{j=0}^\infty(f_k)_\mathbb{C}(T^je\oplus i0)z^j.\]
Let us now show that the real linear functional $f_k\in X^*$ annihilates each $y_n$, and hence all of $Y$.  Indeed, for any $\mu\in\mathbb{R}$ satisfying $|\mu|>r(T)$, the Neumann series gives us the two identites
\begin{multline}\label{id-1}(f_k)_\mathbb{C}\left[(\mu-T_\mathbb{C})^{-1}(e\oplus i0)\right]
=(f_k)_\mathbb{C}\left[\sum_{i=0}^\infty\mu^{-i-1}(T_\mathbb{C})^i(e\oplus i0)\right]
\\=\sum_{i=0}^\infty\mu^{-i-1}(f_k)_\mathbb{C}((T^ie)\oplus i0)
=\sum_{i=0}^\infty\mu^{-i-1}f_k(T^ie)
%=f_k\left[\sum_{i=0}^\infty\mu^{-i-1}T^ie\right]
\end{multline}
and
\begin{multline}\label{id-2}
\sum_{i=0}^\infty\mu^{-i-1}f_k(T^ie)
=\sum_{i=k}^\infty\mu^{-i-1}c_{i-k}
=\sum_{i=k}^\infty(\mu^{-1})^{i+1}c_{i-k}
\\=\sum_{i=0}^\infty(\mu^{-1})^{k+1+i}c_i
=(\mu^{-1})^{k+1}\sum_{i=0}^\infty(\mu^{-1})^ic_i
=\mu^{-k-1}F(\mu^{-1})
.\end{multline}
Note that the map $\mu\in\mathbb{C}\mapsto(\mu-T_\mathbb{C})^{-1}e\in X_\mathbb{C}$ is analytic on $\rho_\infty(T)$ (cf., e.g., \cite[Corollary 6.7]{AA02}), and hence the so is the map $G_k:\mathbb{C}\to\mathbb{C}$ defined by the rule
\[G_k(\mu)=(f_k)_\mathbb{C}\left[(\mu-T_\mathbb{C})^{-1}e\right].\]
As the real values of modulus $>r(T)$ have an accumulation point in $\rho_\infty(T)$, by the Identity Theorem it follows from \eqref{id-1} and \eqref{id-2} that
\[G_k(\mu)=\mu^{-1-k}F(\mu^{-1})\]
for all $\mu\in\rho_\infty(T)$.  In particular, from \eqref{id-1} and the Neumann series for real inverses we now have
\begin{multline*}f_k(y_n)
=f_k\left[(\lambda_n^{-1}-T)^{-1}e\right]
=f_k\left[\sum_{i=0}^\infty(\lambda_n^{-1})^{-i-1}T^ie\right]
=\sum_{i=0}^\infty(\lambda_n^{-1})^{-i-1}f_k(T^ie)
\\=(f_k)_\mathbb{C}\left[(\lambda_n^{-1}-T_\mathbb{C})^{-1}(e\oplus i0)\right]
=G_k(\lambda_n^{-1})=\lambda_n^{1+k}F(\lambda_n)
=0
\end{multline*}
for all $n\in\mathbb{N}_0$.\end{proof}

\begin{theorem}\label{real-3.3}Let $X$ be an infinite-dimensional (real or complex) Banach space and $T\in\mathcal{B}(X)$.  Suppose there exists an infinite chain
\[V_1\supsetneq V_2\supsetneq V_3\supsetneq\cdots\]
of closed subspaces of $X$ satisfing $TV_n\subseteq V_{n+1}$ for all $n\in\mathbb{N}$.  Then $T$ admits an AIHS of defect $\leq 1$.\end{theorem}

\begin{proof}We will follow the proof of \cite[Proposition 3.3]{SW14}, which gives the complex case.  So, we need only consider the case where $X$ is real.

Without loss of generality assume $V_1=X$.  As $V_\infty:=\bigcap_{n=1}^\infty V_n$ is an infinite-codimensional closed $T$-invariant subspace, we may assume it is finite-dimensional, and decompose $X=V_\infty\oplus W$ for some finite-codimensional closed subspace $W$ of $X$.  Write $P_W\in\mathcal{B}(X)$ for the bounded linear projection onto $W$ along $V_\infty$, and define $S:=P_WT|_W\in\mathcal{B}(W)$.

By \cite[Lemma 3.2]{SW14}, $W_n:=P_WV_n=W\cap V_n$ for each $n\in\mathbb{N}$.  Furthermore, these spaces have the following obvious properties.
\begin{itemize}\item[(a)]  $W_{n+1}\subseteq W_n$ for each $n\in\mathbb{N}$;
\item[(b)]  $\bigcap_{n=1}^\infty W_n=W\cap(\bigcap_{n=1}^\infty V_n)=\{0\}$; and
\item[(c)]  $SW_n=P_WT(W\cap V_n)\subseteq W\cap V_{n+1}=W_{n+1}$ for each $n\in\mathbb{N}$.\end{itemize}
From this we can deduce that $\sigma_p(S)\cap\mathbb{R}\subseteq\{0\}$, as if $\lambda$ were a nonzero real eigenvalue of $S$ with corresponding eigenvalue $w\in W=W_1$, we would then have $w=\lambda^{-n}S^nw\in W_n$ for every $n\in\mathbb{N}$, contradicting (b) above.  Also, by the $T$-invariance of $V_\infty$, every AIHS $Y$ under $S$ induces an AIHS $Y+V_\infty$ under $T$, with no larger defect.  Notice also that $0\in\sigma(S)$ and thus, by Theorem \ref{boundary-real-spectrum}, we may assume that $\sigma_\mathbb{R}(S)=\{0\}$.

We claim that there exists a minimal orbit under $S$, which is sufficient by Theorem \ref{minimal-orbit}.  Indeed, by \cite[Lemma 2.11]{SW14} we may assume that $\mathcal{N}(S^n)$ is finite-dimensional for each $n\in\mathbb{N}$.  As any proper closed subspace of a normed linear space is nowhere dense, the Baire Category Theorem ensures that the set $A:=W\setminus\bigcup_{n=1}^\infty\mathcal{N}(S^n)$ is nonempty.  Select any $e\in A$ and fix some $n\in\mathbb{N}$.  By property (b) we can find minimal $k_n\in\mathbb{N}$ such that $S^ne\notin W_{k_n}$.  By minimality, $S^ne\in W_{k_n-1}$, and properties (a) and (c) therefore force $S^j(S^ne)\in W_{k_n}$ for each $j\in\mathbb{N}$.  As $W_{k_n}$ is closed, this means $S^ne\notin[S^je]_{j=n+1}^\infty$, and we are done.\end{proof}

\begin{lemma}\label{RDR}Let $X$ be a real Banach space and $T\in\mathcal{B}(X)$.  Then
\[\sigma_p(T^*)\cap\mathbb{R}=\left\{\lambda\in\mathbb{R}:\overline{(\lambda-T)X}\neq X\right\}.\]\end{lemma}

\begin{proof}Note that
\[\overline{S_\mathbb{C}X_\mathbb{C}}=\overline{SX\oplus iSX}=\overline{SX}\oplus i\overline{SX}\]
for any operator $S\in\mathcal{B}(X)$, and $(\lambda-T)_\mathbb{C}=\lambda-T_\mathbb{C}$ for any $\lambda\in\mathbb{R}$.  Hence, we can apply the well-known identity
\[\sigma_p(T^*)=\sigma_p((T_\mathbb{C})^*)=\left\{\lambda\in\mathbb{C}:\overline{(\lambda-T_\mathbb{C})X_\mathbb{C}}\neq X_\mathbb{C}\right\}\]
(cf., e.g., \cite[Theorem 6.19]{AA02}).\end{proof}

\begin{proposition}\label{spectrum-restricted-operator}Let $X$ be a (real or complex) Banach space, and suppose $W$ is a $T$-invariant finite-codimensional subspace of $X$ for some $T\in\mathcal{B}(X)$.  Decompose $X=W\oplus E$ for some finite-dimensional subspace $E$ of $X$, and let $P_E\in\mathcal{B}(X)$ denote the bounded linear projection onto $E$ along with $\mathcal{N}(P_E)=W$.  Then the following assertions all hold.
\begin{itemize}\item[(i)]  $\sigma(T|_W)\subseteq\sigma(T)$,
\item[(ii)]  $\sigma(T)\setminus\sigma(T|_W)\subseteq\sigma_p(P_ET)$, and
\item[(iii)]  $\sigma(T)\setminus\sigma(T|_W)$ has finite cardinality.\end{itemize}\end{proposition}

\begin{proof}Thanks go to Robert Israel for suggesting the following proof of (i).

Note that if $X$ is a real Banach space then $W_\mathbb{C}\subseteq X_\mathbb{C}$ is invariant under $T_\mathbb{C}$ and $S_\mathbb{C}=T_\mathbb{C}|_{W_\mathbb{C}}$.  Thus we can assume $X$ is complex.

(i)  Select $\lambda\in\rho(T)$ and notice that
\[X=(\lambda-T)X=(\lambda-T)W+(\lambda-T)E\subseteq W+(\lambda-T)E\subseteq X\]
so that
\[\text{dim}(X/W)\leq\text{dim}(X/(\lambda-T)W)\leq\text{dim}((\lambda-T)E)\leq\text{dim}(E)=\text{dim}(X/W).\]
In particular, $(\lambda-T)W$ is a closed subspace of $W$ with the same codimension as $W$ in $X$, whence it follows that $(\lambda-T)W=W$.  Hence \[(\lambda-T)^{-1}W=(\lambda-T)^{-1}(\lambda-T)W=W\]
so that $W$ is invariant under $(\lambda-T)^{-1}$.  It follows that $(\lambda-T)^{-1}|_W\in\mathcal{B}(W)$ is an inverse for $\lambda-T|_W$.

(ii)  Consider the case where $\lambda-T$ is surjective, an recall (cf., e.g., \cite[Theorem 6.20]{AA02}) that whenever $\lambda-T$ is surjective but not invertible, $\lambda\in\sigma_p(T)$.  So, for some $w\in W$ and $e\in E$ with $w+e\neq 0$ we have $(\lambda-T)(w+e)=0$, and hence, by $T$-invariance of $W$,
\[0=P_E(\lambda-T)(w+e)=(\lambda-P_ET)e.\]
Since $\lambda\notin\sigma(T|_W)$ we must have $e\neq 0$.  Hence, $\lambda\in\sigma_p(P_ET)$.

Thus we can assume $\lambda-T$ is not surjective.  In this case, we cannot have $P_E(\lambda-T)X=E$, since that would give us
\begin{multline*}(\lambda-T)X=(1-P_E)(\lambda-T)X+E
\\\supseteq(1-P_E)(\lambda-T)W+E=(1-P_E)W+E=W+E=X,\end{multline*}
contradicting the fact that $\lambda-T$ is not surjective.  Thus, $P_E(\lambda-T)|_E$ has rank $<\text{dim}(E)$.  By the rank-nullity theorem this means $\lambda\in\sigma_p(P_ET)$.

(iii)  Since $P_ET$ is finite-rank, $\sigma_p(P_ET)$ has finite cardinality.  Now apply (ii).\end{proof}

Observe the following corollary to \cite[Lemma 3.4]{SW14}.

\begin{lemma}\label{real-3.4}Let $X$ be an infinite-dimensional real Banach space and $T\in\mathcal{B}(X)$.  For $\mu\in\mathbb{R}$ we define
\[V=\overline{(\mu-T)X}\;\;\;\text{ and }\;\;\;S=T|_V\in\mathcal{B}(V).\]
Then
\[\sigma_p(S^*)\subseteq\sigma_p(T^*).\]\end{lemma}

\begin{proof}Observe that
\[(\mu-T_\mathbb{C})X_\mathbb{C}=(\mu-T)X\oplus i(\mu-T)X\]
and hence
\[\overline{(\mu-T_\mathbb{C})X_\mathbb{C}}=\overline{(\mu-T)X}\oplus i\overline{(\mu-T)X}.\]
Now we apply \cite[Lemma 3.4]{SW14} to get our result.\end{proof}

\begin{remark}\label{real-3.4-polynomial}By successively applying Lemma \ref{real-3.4}, we have the same conclusion $\sigma_p(S^*)\subseteq\sigma_p(T^*)$ whenever $V=\overline{p(T)X}$ and $S=T|_V\in\mathcal{B}(V)$ for some polynomial $p(t)\in\mathbb{R}[t]$ with real coefficients.\end{remark}

\begin{theorem}\label{real-3.5}Let $X$ be an infinite-dimensional real Banach space, and suppose $T\in\mathcal{B}(X)$ satisfies either
\begin{equation}\label{real-spectrum}\sigma(T)\subseteq\mathbb{R}\;\;\;\text{ or }\end{equation}
\begin{equation}\label{real-cardinality}\sigma_\mathbb{R}(T)\text{ has infinite cardinality}.\end{equation}
Then at least one of the following conditions holds.
\begin{itemize}\item[(i)]  $T$ admits an AIHS of defect $\leq 1$; or
\item[(ii)]  for any finite subset $A$ of $\mathbb{R}$ there exists a finite-codimensional closed $T$-invariant subspace $W$ of $X$ such that the operator $S=T|_W\in\mathcal{B}(W)$ satisfies
\[\sigma_p(S^*)\subseteq\sigma_p(T^*)\setminus A\;\;\;\text{ and }\;\;\;\sigma(S)\subseteq\sigma(T).\]
Furthermore, the set $\sigma(T)\setminus\sigma(S)$ has finite cardinality, and whichever conditions \eqref{real-spectrum} and \eqref{real-cardinality} hold for $T$ also hold for $S$.\end{itemize}\end{theorem}

\begin{proof}We follow the proof of \cite[Proposition 3.5]{SW14}.

Fix $\lambda\in\sigma_p(T^*)\cap A$.  If any of the $T$-invariant subspaces $V_n:=\overline{(\lambda-T)^nX}$, $n\in\mathbb{N}$, have finite dimension then we get (i) by \cite[Lemma 2.11]{SW14}.  Thus we may assume $(V_n)_{n=1}^\infty$ all have finite codimension in $X$, as otherwise we have (i) anyway.  By Remark \ref{real-3.4-polynomial}, the operators $S_n:=T|_{V_n}\in\mathcal{B}(X)$ all satisfy $\sigma_p(S_n^*)\subseteq\sigma_p(T^*)$, $n\in\mathbb{N}$.  Consider the case where $\lambda\in\bigcap_{n=1}^\infty\sigma_p(S_n^*)$.  Then by Lemma \ref{RDR}, $(V_n)_{n=1}^\infty$ is strictly decreasing (under the $\supsetneq$ relation).  Now apply Theorem \ref{real-3.3} to get (i).

Thus we may assume $\sigma_p(S_n^*)\subseteq\sigma_p(T^*)\setminus\{\lambda\}$ for some $n\in\mathbb{N}$.  Note that if either condition \eqref{real-spectrum} or \eqref{real-cardinality} holds for $T$ then the same condition holds for $S_n$ by Proposition \ref{spectrum-restricted-operator}.  Thus, if $\sigma_p(S_n^*)\cap A=\emptyset$ then we have (ii).  Otherwise, we can repeat the process in the previous paragraph to eliminate another element of $A$.  As $A$ is finite, we cannot keep losing elements of $A$ indefinitely, eventually giving us (i) or (ii).\end{proof}

\begin{remark}\label{SS-remark}In Theorem \ref{real-3.5} above, we can replace the condition that $T$ satisfies either \eqref{real-spectrum} or \eqref{real-cardinality} with the condition that $T$ is strictly singular.  This is because the restriction of a strictly singular operator to an invariant subspace is again strictly singular, and hence contains zero in its spectrum.\end{remark}

It was observed in \cite{SW16} that every strictly singular operator acting on an infinite-dimensional complex Banach space admits an AIHS of defect $\leq 1$.  We can now extend this result to the real reflexive setting.

\begin{corollary}Let $X$ be an infinite-dimensional real Banach space, and let $T\in\mathcal{B}(X)$.  If $T$ is strictly singular then $T^*$ admits an AIHS of defect $\leq 1$.  If furthermore $X$ is reflexive then so does $T$.\end{corollary}

\begin{proof}By the spectral theorem for strictly singular operators (cf., e.g., \cite[Theorem 7.11]{AA02}), $T$ has countable spectrum so that $0\in\partial\sigma_\mathbb{R}(T)$.  Without loss of generality, let us pass if necessary to the restriction of $T$ to a finite-codimensional $T$-invariant closed subspace in order to remove zero from the point spectrum of $T^*$, which we do via Theorem \ref{real-3.5} and Remark \ref{SS-remark}.  Of course, $T$ is still strictly singular so that zero still lies in $\partial\sigma_\mathbb{R}(T)$ and hence also $\partial\sigma_\mathbb{R}(T^*)$.  Theorem \ref{boundary-real-spectrum} now gives us an AIHS of defect $\leq 1$ for $T^*$.  In case $X$ is reflexive, Theorem \ref{dual-AIHS} gives us an AIHS of defect $\leq 1$ for $T$.\end{proof}

\begin{lemma}\label{dual-restriction}Let $X$ be an infinite-dimensional (real or complex) Banach space and $T\in\mathcal{B}(X)$.  Suppose $W$ is a $T$-invariant finite-codimensional closed subspace of $X$, and write $S=T|_W\in\mathcal{B}(W)$.  If $S^*$ admits an AIHS of defect $\leq d$, $d\in\mathbb{N}_0$, then so does $T^*$.\end{lemma}

\begin{proof}Let us decompose $X=W\oplus E$ for some finite-dimensional subspace $E$ of $X$. Denote by $P_W\in\mathcal{B}(X)$ the bounded linear projection onto $W$ along $E$, and by $E^\perp\subseteq X^*$ the annihilator of $E$.  Define the bounded linear map $j:E^\perp\to W^*$ by the rule
\[jx^*=x^*|_W\;\;\;\text{ for all }x^*\in E^\perp,\]
which admits a bounded linear inverse $j^{-1}:W^*\to E^\perp$ given by the rule
\[j^{-1}w^*=w^*\circ P_W\;\;\;\text{ for all }w^*\in W^*.\]

Set $R:=j^{-1}\circ S^{*}\circ j:E^{\perp}\to E^{\perp}$, and let $H$ be an AIHS for $S^{*}$ with error $M$, $\text{dim}(M)\leq d$. Put $G:=j^{-1}H$ and $N:=j^{-1}M$, which are closed subspaces of $E^{\perp}$ and hence also of $X^{*}$.  Then for any $x^{*}\in E^{\perp}$ and any $x\in X$ we now have
\begin{eqnarray*}
  ((P_W^{*}T^{*})x^{*})(x) & = & x^{*}(TP_Wx) \hskip 4cm  (\mbox{note } P_Wx\in W)\\
    &=& x^{*}(SP_Wx)=(jx^{*})(SP_Wx) \hskip 1cm (\mbox {note } (jx^{*})\in W^{*} )\\
    &=& (S^{*}(jx^{*}))(P_Wx) = ((j^{-1}\circ S^{*}\circ j)(x^{*}))(x) \\
    &=& (Rx^{*})(x)
\end{eqnarray*}
so that $P_W^{*}T^{*}\equiv R$ on $E^\perp$.  Now we have
\[(P_W^{*}T^{*})G=RG=(j^{-1}\circ S^{*}\circ j)(j^{-1}H)\\=j^{-1}(S^{*}H)\subseteq j^{-1}(H+M)=G+N\]
and hence
\begin{equation}\label{dual-relation}T^{*}G=(P_W^{*}T^{*})G+((1-P_W^{*})T^{*})G\subseteq G+N+((1-P_W^{*})T^{*})G.\end{equation}

Now, $1-P_W^*$ is a projection onto $W^{\perp}$, so that $(1-P_W^*)T^*(G)\subseteq W^{\perp}$.  Since $W$ is an finite-codimensional invariant subspace for $T$, we must have $W^\perp$ a finite-dimensional invariant subspace for $T^*$. Put  $G_0:=G+W^{\perp}$ which is a halfspace in $X^*$ satisfying
\begin{multline*}
  T^*G_0 = T^{*}(G+W^{\perp}) \subseteq  T^*(G)+T^*(W^{\perp})
  \\ \subseteq G+N+((1-P_W^{*})T^{*})(G)+W^{\perp}
 \subseteq  G+N+W^{\perp}=G_0+N.
\end{multline*}
As $\text{dim}(N)=\text{dim}(M)$, the space $G_0$ is an AIHS under $T^*$ with defect $\leq d$.\end{proof}

\begin{lemma}\label{infinitely-many-eigenvalues}Let $X$ be an infinite-dimensional real Banach space, and let $T\in\mathcal{B}(X)$.  If $\sigma_p(T)\cap\sigma_p(T^*)\cap\mathbb{R}$ is infinite then $T$ admits an IHS.\end{lemma}

\begin{proof}We follow the proof of \cite[Theorem 2.7]{PT13}.  Observe that for any $\lambda,\mu\in\sigma_p(T)\cap\sigma_p(T^*)\cap\mathbb{R}$ we can find a $\lambda$-eigenvector $x\in X$ and a $\mu$-eigenvector $f\in X^*$, and these must satisfy
\[\lambda f(x)=f(\lambda x)=f(Tx)=(T^*f)(x)=\mu f(x).\]
Thus, we can find infinitely many linearly independent $T^*$-eigenvectors annihilating infinitely many linearly independent $T$-eigenvectors $(x_n)_{n=1}^\infty$.  This gives us an IHS $[x_n]_{n=1}^\infty$.\end{proof}

\begin{theorem}\label{real-3.6}Let $X$ be an infinite-dimensional (real or complex) Banach space and $T\in\mathcal{B}(X)$ a bounded linear operator satisfying either of the conditions \eqref{real-spectrum} or \eqref{real-cardinality} from Theorem \ref{real-3.5}.  Then $T^*$ admits an AIHS of defect $\leq 1$.\end{theorem}

\begin{proof}Let us define $W_0=X$ and $S_0=T$, and select $\lambda_0\in\partial\sigma_\mathbb{R}(S_0)$.  As $\sigma_\mathbb{R}(S_0)=\sigma_\mathbb{R}(S_0^*)$, we may assume by Theorems \ref{boundary-real-spectrum} and \ref{dual-AIHS} that $\lambda_0\in\sigma_p(S_0)\cap\sigma_p(S_0^*)$.  By Theorems \ref{real-3.5} and \ref{dual-AIHS} we may assume the existence of a finite-codimensional $T$-invariant subspace $W_1\subseteq W_0$ such that $S_1:=S_0|_{W_1}\in\mathcal{B}(W_1)$ satisfies $\sigma_p(S_1^*)\subseteq\sigma_p(T)\setminus\{\lambda_0\}$ and $\sigma(S_1)\subseteq\sigma(T)$, and such that whichever conditions \eqref{real-spectrum} or \eqref{real-cardinality} hold for $T$ also hold for $S_1$.

Now select $\lambda_1\in\partial\sigma_\mathbb{R}(S_1)$, which we may again by Theorems \ref{boundary-real-spectrum} and \ref{dual-AIHS}, together with Lemma \ref{dual-restriction}, assume is an eigenvalue for both $S_1$ and $S_1^*$.  Again we apply Theorems \ref{real-3.5} and \ref{dual-AIHS} so that we can assume $W_2$ is a closed finite-codimensional $S_1$-invariant subspace with
\[\sigma_p(S_2^*)\subseteq\sigma_p(S_1^*)\setminus\{\lambda_2\}\subseteq\sigma_p(T^*)\setminus\{\lambda_1,\lambda_2\},\]
and again so that whichever conditions \eqref{real-spectrum} or \eqref{real-cardinality} hold for $T$ also hold for $S_2$.

We keep going until we have a sequence $(\lambda_n)_{n=0}^\infty$ of distinct real eigenvalues under each respective operator $S_n$.  As $S_n=T|_{W_n}$ for each $n\in\mathbb{N}_0$ we get $(\lambda_n)_{n=0}^\infty\subseteq\sigma_p(T)$.  Since $\sigma_p(S_n^*)\subseteq\sigma_p(T^*)$ for each $n\in\mathbb{N}_0$, we also have $(\lambda_n)_{n=0}^\infty\subseteq\sigma_p(T^*)$.  Applying Lemma \ref{infinitely-many-eigenvalues} completes the proof.\end{proof}

Thus we have the following immediate corollaries.

\begin{corollary}\label{real-reflexive}Let $X$ be an infinite-dimensional (real or complex) reflexive space, and let $T\in\mathcal{B}(X)$.  If $T$ satisfies either of the conditions \eqref{real-spectrum} or \eqref{real-cardinality} from Theorem \ref{real-3.5} then it admits an AIHS of defect $\leq 1$.\end{corollary}

\begin{proof}Apply Theorems \ref{dual-AIHS} and \ref{real-3.6}.\end{proof}

\begin{corollary}\label{selfadjoint}Every selfadjoint operator acting on an infinite-dimensional (real or complex) Hilbert space admits an AIHS of defect $\leq 1$.\end{corollary}

In \cite[\S4]{SW16} it was shown that if $X$ is an infinite-dimensional complex Banach space then for every $T\in\mathcal{B}(X)$ there exists a rank-1 operator $F\in\mathcal{B}(X)$ such that for each $\varepsilon>0$ there there is a nuclear operator $N\in\mathcal{B}(X)$ of norm $<\varepsilon$, and such that $T+F+N$ admits an IHS.  In particular, every operator acting on $X$ admits an EIHS.  We extend this result to the real setting as follows.

\begin{theorem}Let $X$ be an infinite-dimensional real Banach space and $T\in\mathcal{B}(X)$.  If either condition \eqref{real-spectrum} or \eqref{real-cardinality} from Theorem \ref{real-3.5} holds, then there exists a rank-1 operator $F\in\mathcal{B}(X)$ such that for every $\varepsilon>0$ there there is a nuclear operator $N\in\mathcal{N}(X)$ with $\|N\|<\varepsilon$ and such that $T+F+N$ admits an IHS. In particular, $T$ admits an EIHS.\end{theorem}

\begin{proof}Consider the case where $\partial\sigma_\mathbb{R}(T)$ is infinite.  Then we may assume by Theorem \ref{boundary-real-spectrum} that $\sigma_p(T)\cap\mathbb{R}$ is also infinite and hence, shifting if necessary, by Lemma \ref{infinitely-many-eigenvalues} that $0\in(\partial\sigma_\mathbb{R}(T))\setminus\sigma_p(T^*)$.  Next, consider the case where $\partial\sigma_\mathbb{R}(T)$ is finite.  If $\sigma_\mathbb{R}(T)$ is finite then condition \eqref{real-spectrum} must hold, which means we can apply Theorem \ref{real-3.5} to again to obtain $0\in(\partial\sigma_\mathbb{R}\sigma(T))\setminus\sigma_p(T^*)$.  Otherwise $\sigma_\mathbb{R}(T)$ is infinite and $\partial\sigma_\mathbb{R}(T)$ is finite, which means we can pass to a shifted operator if necessary so that $0\in\partial\sigma_\mathbb{R}(T)$ and is a limit point of $\sigma_\mathbb{R}(T)$.  By Theorem \ref{real-3.5} we can find a $T$-invariant finite-codimensional closed subspace $W$ of $X$ such that $0\notin\sigma_p(T|_W^*)$.  By Proposition \ref{spectrum-restricted-operator} we still have $0\in\partial\sigma_\mathbb{R}(T)$ so that, again, by passing to $T|_W$ if necessary, we may assume that $0\in(\partial\sigma_\mathbb{R}(T))\setminus\sigma_p(T^*)$.

Thus, in all cases, the relations $\partial\sigma_\mathbb{R}(T)\subseteq\partial\sigma(T)\subseteq\sigma_{su}(T)$ (cf., e.g., \cite[Theorem 2.42]{Ai04}) give us $0\in\sigma_{su}(T)\cap\sigma_p(T)$.  So, at this point we may follow the proof of \cite[Theorem 4.2]{SW16}.  We claim that $T|_U$ is not bounded below for any finite-codimensional closed subspace $U$ of $X$.  Otherwise, towards a contradiction $T|_U$ and hence also $T$ have closed range by \cite[Theorem 2.5]{AA02}.  Since $0\notin\sigma_p(T^*)$, by Lemma \ref{RDR} we must have $\overline{TX}=X$ and hence $TX=X$.  This contradicts the fact that $0\in\sigma_{su}(T)$, and so the claim is proved.  Note that in the last paragraph if the proof to \cite[Theorem 4.2]{SW16} it was shown that whenever this condition holds---that is, whenever $T|_U$ fails to be bounded below for finite-codimensional closed subspaces $U$ of $X$---then $T+N$ admits an IHS for some nuclear operator $N\in\mathcal{B}(X)$ of norm $<\varepsilon$.  This completes the proof.
%Now, by \cite[Proposition 2.c.4]{LT77}, for $\varepsilon>0$ there exists an infinite-dimensional subspace $V$ of $X$ such that $T|_X$ is compact and $\|T|_V\|<\varepsilon$.  We may assume that $V=[v_n]_{n=1}^\infty$ for some seminormalized basic sequence $(v_n)_{n=1}^\infty\subseteq X$, and that $Tv_n$ converges to some $v\in V$.  In particular, we can pass to a subsequence if necessary so that $y_n=v_{2n}-v_{2n-1}$ forms a seminormalized basic sequence, $Y:=[y_n]_{n=1}^\infty$ is a halfspace, and $y_n\to 0$ as fast as we like.  Now apply Proposition \ref{nuclear-extension} to complete the proof.
\end{proof}

\end{document}